\documentclass[12pt]{amsart}

\usepackage[margin=1.15in]{geometry}

\usepackage{amsmath,amscd,amssymb,amsfonts,latexsym,wasysym, mathrsfs, mathtools,hhline,xcolor}
\usepackage[all, cmtip]{xy}
\usepackage{csquotes}
\usepackage{url} 

\definecolor{hot}{RGB}{65,105,225}

\usepackage[pagebackref=true,colorlinks=true, linkcolor=hot ,  citecolor=hot, urlcolor=hot]{hyperref}%make all the references and links clickable

\usepackage{ textcomp }
\usepackage{ tipa }
\usepackage{tikz}
\usetikzlibrary{arrows.meta}

\usepackage{graphicx}
\usepackage[shortlabels]{enumitem}

\newlist{kotschickconditions}{enumerate}{1}
\setlist[kotschickconditions]{
  label=(\Alph*),
  ref=(\Alph*),
  leftmargin=*,
  itemsep=0pt
}

\newcommand{\condref}[1]{%
  \hyperref[cond:#1]{\ref*{cond:#1}}%
}

\theoremstyle{plain}
\newtheorem{theorem}{Theorem}%%[section]
\newtheorem{proposition}[theorem]{Proposition}

\newtheorem{lm}[theorem]{Lemma}

\newtheorem{corollary}[theorem]{Corollary}

\newtheorem{lemma}[theorem]{Lemma}
\newtheorem{thrm}[theorem]{Theorem}

\theoremstyle{definition}

\newtheorem{defn}[theorem]{Definition}
\newtheorem{question}[theorem]{Question}
\newtheorem{remark}[theorem]{Remark}

\newtheorem{ex}[theorem]{Example}
\newtheorem*{ex*}{Example}

\newtheorem*{conjecture*}{Conjecture}
\def\be{\begin{equation}}
\def\ee{\end{equation}}

\def\bt{\begin{thrm}}
\def\et{\end{thrm}}

\def\bc{\begin{cor}}
\def\ec{\end{cor}}

\def\br{\begin{rmk}}
\def\er{\end{rmk}}

\def\bp{\begin{prop}}
\def\ep{\end{prop}}

\def\bl{\begin{lm}}
\def\el{\end{lm}}

\def\bex{\begin{ex}}
\def\eex{\end{ex}}

\def\bd{\begin{defn}}
\def\ed{\end{defn}}

\DeclareMathOperator{\codim}{codim}              % codim
\DeclareMathOperator{\id}{id}                    % id

\DeclareMathOperator{\Alb}{Alb}

\DeclareMathOperator{\Bl}{Bl}

\def\bP{\mathbb{P}}
\def\bA{\mathbb{A}}

\def\cO{\mathcal{O}}

\def\bZ{\mathbb{Z}}

\def\bN{\mathbb{N}}

\def \Z{\mathbb Z}
\def \R{\mathbb R}
\def \Q{\mathbb Q}

\def \C{\mathbb C}

 \title[Invisible singularities]{Invisible singularities  in \\   complex algebraic geometry}

\author{Maur\'icio Corr\^ea}
\address{Universit\`a degli Studi di Bari Aldo Moro, Dipartimento di Matematica, Via Edoardo Orabona 4, 70125 Bari, Italy.}
\email{mauricio.barros@uniba.it}

\author{J\'anos Koll\'ar}
\address{Princeton University, Department of Mathematics, Fine Hall, Washington Road, Princeton, NJ 08544-1000, USA.}
\email{kollar@math.princeton.edu}

\author{Stefan Schreieder}
\address{Leibniz Universit\"at Hannover, Institut f\"ur Algebraische Geometrie, Welfengarten 1, 30167 Hannover, Germany.}
\email{schreieder@math.uni-hannover.de}

\author{Botong Wang}
\address{Department of Mathematics, University of Wisconsin-Madison,  480 Lincoln Drive, Madison, WI 53706-1388, USA.}
\email{wang@math.wisc.edu}
\date{\today}

\hypersetup{
  pdftitle={Invisible singularities  in complex algebraic geometry},
  pdfauthor={Maur\'icio Corr\^ea, J\'anos Koll\'ar, Stefan Schreieder and Botong Wang},
  pdfkeywords={Topology of varieties, Vanishing cycles, Fiber bundles, Zariski conjecture, Kotschick conjecture, Universal covers of varieties, Fibrations over the circle, Holomorphic one-forms, Albanese morphisms.} 
}
 


\begin{document}

\begin{abstract} 
We construct morphisms between smooth complex projective varieties that have singular fibers, but look topologically smooth.  
We use this to give negative answers to the following three conjectures and questions: the smoothness conjecture of Fern\'andez~de~Bobadilla and Koll\'ar \cite{FdBK}, a question of Koll\'ar and Pardon on universal covers \cite{KP}, and Kotschick's conjecture on 1-forms \cite{K}. 
\end{abstract}

\maketitle
\section{Introduction} 
A central question of equisingularity theory is to understand the relationship between the algebraic, analytic and topological versions. After giving full answers for curves, Zariski outlined a  series of higher-dimensional questions in \cite{Z}. 
The general expectation was that the algebraic and topological concepts of equisingularity coincide.

After many positive results \cite{milnor-book, MR399088, teiss-sim, Massey, MR1849319}, Fern\'andez de Bobadilla \cite{FdB}  constructed  families of cubic hypersurfaces  $Y\to C$ over a curve $C$, for which the various  notions 
of equisingularity introduced by  Zariski and Massey are not equivalent.
 In these, and the later examples in \cite{FdBK}, the total space $Y$ is always singular. 

This led to the belief that Zariski's expectation holds for smooth $Y$. 
That is,  if $Y$ is smooth and $g\colon Y\to C$ is proper, then critical points of $g$ are always visible topologically. 
This holds for isolated singularities by  Milnor \cite{milnor-book}, and the local version of the conjecture also holds for non-isolated singularities by  A'Campo \cite{ACampo1973}; cf.~\cite[Propositions 13 and 14]{FdBK}.
See \cite{FdBK, LMW, MR4461010, MR4768419} for further discussions and  results. 
 
Nonetheless, our main result shows that the aforementioned expectation fails.

\begin{theorem} \label{thm:main.invis.thm} 
There are morphisms $g\colon Y\to \bP^1$ from smooth complex projective varieties $Y$, where $g$ has singular fibers but looks topologically smooth. 
 More precisely,
 \begin{enumerate}[(i)] 
 \item $g$ is a homotopy fiber bundle 
 (see Definition \ref{def:fiber-bundle}), 
 \label{item:main.invis.thm:2}
   \item all the fibers of $g$ are simply connected $PL$-manifolds
that are pairwise homotopy equivalent,
\label{item:main.invis.thm:3}
 \item in the homotopy class of $g$ 
 there is a  locally trivial, $C^\infty$-fiber bundle $\tilde g\colon Y\to \bP^1$ that agrees with $g$ outside a small neighborhood of the singular fibers of $g$,\label{item:main.invis.thm:4}
\item $g$ is not smooth and in fact not a $C^0$-fiber bundle.
\label{item:main.invis.thm:5}
 \end{enumerate} 
\end{theorem}

Theorem \ref{thm:main.invis.thm} disproves the smoothness conjecture of Fern\'andez de Bobadilla--Koll\'ar \cite[Conjecture 3]{FdBK}, which is a special case of Zariski's equisingularity philosophy \cite{Z}.

\begin{remark} Our construction of these $g\colon Y\to \bP^1$ is completely explicit.

The smallest example we have has $\dim Y=5$, 
the critical locus of $g$ has dimension 1, and, near any critical point, one can write $g$ as $(x_1, \dots, x_5)\mapsto x_1^2+x_2^2+x_3^2+x_4^3$ (in  suitable local analytic coordinates), see Remark \ref{rem:terminal}. 
Moreover, $g$ is birationally equivalent to the projection $\bP^4\times \bP^1\to \bP^1$, see Example \ref{ex:RC-4-fold-bundle}.

One can deduce from \cite{HS} that there are no such examples with $\dim Y\leq 3$, see Theorem \ref{thm:3-folds} below.  The question is open in dimension 4.

We construct the example in Theorem~ \ref{thm:main.invis.thm} by choosing the singularities of $g$ carefully. 
We first show that $g$ is a $\Z$-homology fiber bundle via a computation involving vanishing cycles.
For  simply connected fibers, this
implies item \ref{item:main.invis.thm:2} by Whitehead's theorem.
Item \ref{item:main.invis.thm:3} follows  from the works of Brieskorn \cite{Brieskorn} and Milnor \cite{milnor-book}.
 
It turns out that 
for simply connected fibers 
\ref{item:main.invis.thm:2} 
implies \ref{item:main.invis.thm:4}. Here we rely on Smale's h-cobordism theorem 
and Cerf's pseudo-isotopy theorem \cite{Cerf}; see Proposition \ref{prop:smooth-fibration-over-Sigma} below. 
However, these provide very little information about $\tilde g$ beyond its existence. 
 \end{remark}

We also get negative answers to  \cite[Question~27]{KP} and to
Kotschick's conjecture, first asked around 2013, but published only later in  \cite[Question 15]{K} and \cite{S,HS,SY,pietig}; cf.~Corollaries \ref{cor:Kotschick} and \ref{cor:KP} below. 

\begin{corollary}\label{cor:consequences}
Let $E$ be an elliptic curve and $E\to \bP^1$ a double cover unramified over the critical values of $g\colon Y\to \bP^1$ from Theorem \ref{thm:main.invis.thm}. 
Set $f\colon X\coloneq Y\times_{\bP^1}E \to E$. 
Then
\begin{enumerate}[(i)]
\item  the universal
cover $\widetilde X$ has the homotopy type of a finite CW complex,
but $f$ is not smooth;\label{item:thm:consequences:2}
\item  every nonzero class
$u\in H^1(X,\mathbb R)$ can be represented by a real, closed one-form
without zeros, yet the harmonic representative of every such class has zeros. In particular, $X$ admits no holomorphic one-form
without zeros.\label{item:thm:consequences:3}
\end{enumerate}
\end{corollary}
 
\subsection*{Outline of the paper}
We always work over the field of complex numbers.

In Section \ref{sec:homology-examples}, we construct flat projective morphisms $f\colon X\to \Delta$ over the disc $\Delta$ with smooth total space $X$ such that $f$ is a $\Z$-homology fiber bundle but not a homotopy fiber bundle.
In Section \ref{sec:smooth-fibration}, we prove a $C^\infty$-smoothing result up to homotopy for morphisms to Riemann surfaces, which shows in particular that item \ref{item:main.invis.thm:2} in Theorem \ref{thm:main.invis.thm} implies item \ref{item:main.invis.thm:4}.
In Section \ref{sec:main.invis.thm}, we globalize the construction from Section \ref{sec:homology-examples} and combine it with a blow-up construction to prove Theorem~\ref{thm:main.invis.thm}.
In Section~\ref{subsec:thm:3-folds}, we use some of the positive results of \cite{HS} to prove Theorem~\ref{thm:3-folds}, which shows that counterexamples---such as those produced in Theorem~\ref{thm:main.invis.thm}---do not occur in dimension three.

In Section \ref{sec:applications}, we discuss applications to the Koll\'ar--Pardon question and Kotschick's conjecture, and prove items \ref{item:thm:consequences:2} and \ref{item:thm:consequences:3} of Corollary \ref{cor:consequences}.

\section{Homology fiber bundles that are not homotopy fiber bundles}\label{sec:homology-examples}

In this section we discuss the local situation and provide homology fiber bundles over the complex disc that are not homotopy fiber bundles.

\subsection{Homology and homotopy fiber bundles} 
We start by recalling the following.

\begin{defn} \label{def:fiber-bundle}
Let $f\colon M\to N$ be a proper morphism between complex spaces.
It is a \emph{$\Z$-homology fiber bundle} if there is an open covering $N=\bigcup U_i$ such that for each $y\in U_i$, the inclusion $f^{-1}(y)\hookrightarrow f^{-1}(U_i)$ induces an isomorphism on integral homology $H_\ast(f^{-1}(y),\Z)\stackrel{\cong}\to H_\ast(f^{-1}(U_i),\Z)$, or, equivalently, on cohomology  $H^\ast(f^{-1}(U_i),\Z)\stackrel{\cong}\to H^\ast(f^{-1}(y),\Z)$.
This condition is equivalent to requiring that $R^if_\ast \Z$ be a local system for all $i$.

Similarly, $f$ is a \emph{homotopy fiber bundle} if there is a covering as above such that the inclusion $f^{-1}(y)\hookrightarrow f^{-1}(U_i)$ is a homotopy equivalence for all $y\in U_i$.

Clearly, any homotopy fiber bundle is also a $\Z$-homology fiber bundle. 
We will see in Corollary \ref{cor:thm:CK-criterion} that the converse fails in general even if $M$ and $N$ are smooth.
\end{defn}

The following is a direct consequence of Whitehead's theorem.
\begin{lm}\label{cor:whitehead}
Let $f\colon M\to N$ be a proper morphism between  complex spaces.
Assume that $f$ is a $\Z$-homology fiber bundle with simply connected fibers. 
Then 
\begin{enumerate}[(i)]
    \item $f$ is a homotopy fiber bundle;\label{item:cor:whitehead:1}
    \item if $N$ is contractible, then $M$ has the homotopy type of a finite CW complex.\label{item:cor:whitehead:2}
\end{enumerate}
\end{lm}
\begin{proof} 
Let $y\in N$.
Then there is an open neighborhood $U_y\subset N$ of $y$ such that the fiber $M_y\coloneq f^{-1}(y)$ is a deformation retract of $M_{U_y}\coloneq f^{-1}(U_y)$.
Since the fibers of $f$ are simply connected, we can choose an open cover $N=\bigcup_{i\in I} U_i$ such that for all $i\in I$, $U_i$ is contractible and $f^{-1}(U_i)$ is simply connected.
For all $y\in U_i$, the inclusion $M_{y}\hookrightarrow M_{U_i}$ is a map between simply connected and finite CW complexes which induces an isomorphism on homology by the Leray spectral sequence.
Hence, $M_{y}\hookrightarrow M_{U_i}$ is a homotopy equivalence by Whitehead's theorem \cite[Corollary 4.33]{H}.
This concludes the proof of item \ref{item:cor:whitehead:1}.

In order to prove item \ref{item:cor:whitehead:2}, assume that $N$ is contractible.
By van Kampen's theorem, $M$ is then simply connected.
Moreover, 
$M_y\to M$ induces an isomorphism on homology and hence is a homotopy equivalence by Whitehead's theorem \cite[Corollary 4.33]{H}.
In particular, $M$ has the homotopy type of a finite CW complex.
\end{proof}

\subsection{The quotient construction}

Let $Z$ be a complex manifold of dimension $n$, and let
$\pi\colon Z\to \Delta$ be a projective morphism with finitely many
critical points $p_i\in Z$, $i\in I$, all lying over the origin $0\in\Delta$.
Let $M_i$ denote the Milnor fiber of $\pi$ at $p_i$. 
We assume that there is at least one critical point, i.e.~$I\neq \emptyset$.

Let $\gamma$ be an automorphism of $Z$ of  order $1<c<\infty$, that commutes with $\pi$ and
fixes each $p_i$. Let $T_i$ denote the induced action on
$H^{n-1}(M_i,\Z)$ and let $\chi_i(x)$ be its characteristic polynomial.
Let $A$ be an abelian variety and let $\tau\colon A\to A$ be a
translation of order $c$. Set
\begin{align} \label{def:pi-Y}
\pi_Y\colon Y\coloneq (A\times Z)/\langle(\tau,\gamma)\rangle
\longrightarrow \Delta,
\end{align}
where we take the quotient by the diagonal action. Note that $Y$ is
smooth and $\pi_Y$ is projective.

\begin{theorem}\label{thm:CK-criterion}
In the above notation, assume in addition that $\chi_i(1)=\pm1$ for
every $i$. Then $\pi_Y\colon Y\to\Delta$ is a $\Z$-homology fiber bundle but not a homotopy fiber bundle.
\end{theorem}

The smallest such $Z$, to be constructed in  Example~\ref{ex:CK-cubic}, 
has dimension 2. Choosing the dimension of  $A$ suitably gives the following.

\begin{corollary} \label{cor:thm:CK-criterion}
For every  $m\geq 2$
there is a flat, projective morphism $\pi_Y\colon Y\to \Delta$ of smooth analytic varieties of relative dimension $m$, such that $\pi_Y$ is a $\Z$-homology fiber bundle but not a homotopy fiber bundle. 
\end{corollary}

In all the examples constructed below, the fibers $Y_t$ have the property that their fundamental group is an extension of a finite group by $\Z^{2r}$ for some $r\geq1$.

We first record the topological calculation used in the proof.

\begin{lemma} \label{lem:CK-criterion}
Let $n\geq 2$ and let $M$ be a topological space whose only nonzero cohomology groups are $H^0(M,\Z)=\Z$ and $H^{n-1}(M,\Z)\cong\Z^r$.
Let $\gamma$ be an automorphism of $M$ of order $c$, let $T$ be its
induced action on $H^{n-1}(M,\Z)$, and let $\chi_T(x)$ be the
characteristic polynomial of $T$.
Let $A$ be an abelian variety and let $\tau\colon A\to A$ be a
translation of order $c$. Then the projection
\[
(A\times M)/\langle(\tau,\gamma)\rangle\longrightarrow A/\langle\tau\rangle
\]
induces an isomorphism on integral cohomology if and only if
$\chi_T(1)=\pm1$.
\end{lemma}

\begin{proof}
Topologically, an abelian variety of dimension $m$ is
$(S^1)^{2m}$. We may assume that $\tau$ acts on the first factor only.
Then
\[
(A\times M)/\langle(\tau,\gamma)\rangle\simeq
\bigl((S^1\times M)/\langle(\tau,\gamma)\rangle\bigr)
\times (S^1)^{2m-1}.
\]
By the K\"unneth formula, it remains to consider the projection
\[
q\colon (S^1\times M)/\langle(\tau,\gamma)\rangle
\longrightarrow S^1/\langle\tau\rangle.
\]
By assumption, $R^iq_*\Z$ is nonzero only for
$i=0$ and $i=n-1$. The base $S^1/\langle\tau\rangle$ is homeomorphic to $S^1$.
Thus the Leray spectral sequence degenerates at the $E_2$ page.

We have $q_*\Z\cong\Z$, and
$R^{n-1}q_*\Z$ is the local system with fiber
$H^{n-1}(M,\Z)$ and monodromy $T$; denote it by $L_T$. 
A direct computation gives that 
\[
H^0(S^1,L_T)=\ker(T-I)\qquad \text{and}\qquad 
H^1(S^1,L_T)=\operatorname{coker}(T-I).
\]
Thus both groups vanish if and only if $T-I\in\mathrm{GL}_r(\Z)$, which is equivalent to $\det(T-I)=\pm1$, and hence to $\chi_T(1)=\pm1$. 
The assertion follows now from the Leray spectral sequence and the K\"unneth formula.
\end{proof}

\begin{proof}[Proof of Theorem~\ref{thm:CK-criterion}]
Let $p_i\in B_i\subset Z$ be disjoint, $\gamma$-invariant balls around
the critical points $p_i$. Let $B$ be the union of the $B_i$, with
interior $B^\circ$ and boundary ${\partial B}$. For $t\in\Delta$ near $0$, 
$Z_t$ is the union of $B\cap Z_t$ and $Z_t\setminus B^\circ$, meeting
along ${\partial B}\cap Z_t$. Correspondingly, $Y_t$ is the union of
\begin{align} \label{eq:CK-criterion-1}
\bigl(A\times(B\cap Z_t)\bigr)/\langle(\tau,\gamma)\rangle
\quad\text{and}\quad
\bigl(A\times(Z_t\setminus B^\circ)\bigr)/\langle(\tau,\gamma)\rangle,
\end{align}
meeting along
\begin{align} \label{eq:CK-criterion-2}
\bigl(A\times({\partial B}\cap Z_t)\bigr)/\langle(\tau,\gamma)\rangle.
\end{align}
For all sufficiently small $t$, retraction to $Z_0$ gives homeomorphisms between the second spaces in \eqref{eq:CK-criterion-1}, and the same holds for the intersection spaces in \eqref{eq:CK-criterion-2}. 
Since $B_i\cap Z_0$ is contractible, each
$$
\bigl(A\times(B_i\cap Z_0)\bigr)/\langle(\tau,\gamma)\rangle
$$
retracts to $A/\langle\tau\rangle$.

Using the Mayer--Vietoris sequence, it remains to show that the natural
maps
$$
\bigl(A\times(B_i\cap Z_t)\bigr)/\langle(\tau,\gamma)\rangle
\longrightarrow A/\langle\tau\rangle
$$
induce isomorphisms on integral cohomology. 
For $t\neq0$, $B_i\cap Z_t$ is the Milnor fiber $M_i$, and the assertion follows from Lemma \ref{lem:CK-criterion} and the assumption $\chi_i(1)=\pm1$. 
Hence $\pi_Y$ in \eqref{def:pi-Y} is a $\Z$-homology fiber bundle, as we want.

It remains to show that $\pi_Y$ is not a homotopy fiber bundle.
For a contradiction, assume that $\pi_Y$ is a homotopy fiber bundle.
Then the same holds for $A\times Z\to \Delta$, because $A\times Z\to Y$ is finite \'etale.
This is a contradiction, because $Z\to \Delta$ has at least one isolated critical point in the central fiber and so $Z\to \Delta$ is not a $\Q$-homology fiber bundle; the same follows for $A\times Z\to \Delta$ by the K\"unneth formula. 
\end{proof}

\subsection{A family of bielliptic surfaces}
We start with the simplest,  
two-dimensional example.

\begin{ex}\label{ex:CK-cubic}
Set
\begin{align} \label{def:mathcal-X-CK}
S\coloneq \bigl\{zy^2+x^3+tz^3=0\bigr\}
\subset \bP^2_{x,y,z}\times\Delta,
\end{align}
where $t$ is a coordinate on the disc $\Delta$.
Let $\zeta$ be a primitive sixth root of unity and consider the order-six
automorphism
$$
\gamma \colon S\longrightarrow S,
\qquad ([x:y:z],t)\longmapsto([\zeta^2x:\zeta^3y:z],t).
$$
Let $F$ be an elliptic curve and let $\tau\in\operatorname{Aut}(F)$ be  translation by a point of order six.
Set
$$
Y\coloneq
( S\times F)/\langle(\gamma,\tau)\rangle,
$$
and let $f\colon Y\to\Delta$ be the natural projection.

The Milnor fiber of $ S\to\Delta$ can be identified with the
affine elliptic curve
$$
M=\{y^2+x^3+1=0\}\subset\bA^2.
$$
Thus $H^1(M,\Z)\cong\Z^2$, and the induced action has characteristic
polynomial
$$
\chi(x)=x^2-x+1.
$$
In particular, $\chi(1)=1$, and Theorem~\ref{thm:CK-criterion} shows that $f$ is a $\Z$-homology fiber bundle.
The special fiber $Y_0$ is a cuspidal curve bundle over $F/\langle\tau \rangle$, whereas the other fibers are bielliptic surfaces. 
In particular, $f$ is neither smooth nor a $C^0$-fiber bundle. 
It is not a homotopy fiber bundle either: the finite \'etale cover $S\times F\to Y$ is not even a $\Q$-homology fiber bundle over $\Delta$.
\end{ex}

\begin{remark}
The general fiber $Y_t$ of $Y\to \Delta$ in Example~\ref{ex:CK-cubic} is a bielliptic surface of Bagnera--de Franchis type $\Z/6\Z$ (type $7$ in the numbering of \cite[Proposition~1.2]{Serrano}). 
Its integral cohomology is torsion-free (cf.~\cite[Remark~1.6]{Serrano}), given by 
\[
H^k(Y_t,\Z)\cong
\begin{cases}
\Z, & k=0,4,\\
\Z^2, & k=1,2,3.
\end{cases}
\]  
\end{remark}

\subsection{Further examples} \label{subsec:further-examples}

Fix $n\geq2$ and a tuple $\vec c=(c_1,\ldots,c_n)\in\bN^n$ of pairwise coprime integers $c_i\geq2$. 
Set
$c\coloneq\prod_i c_i $,  $a_i\coloneq c/c_i$ and
$$
Z_{\vec{c}}^\circ\coloneq
\left\{
\textstyle\sum_{i=1}^n x_i^{c_i}
=t 
\right\}
\subset\bA^n\times\bA^1_t.
$$ 
There is a $\mu_c$-action on $Z_{\vec c}^\circ$, given by
$$
\gamma_{\vec{c}}\colon (x_1,\cdots,x_n,t) \longmapsto (\zeta^{a_1}x_1,\cdots,\zeta^{a_n}x_n,t),
$$
where $\zeta$ denotes a fixed primitive c-th root of unity.

We check that the assumptions of Lemma~\ref{lem:CK-criterion}  are satisfied.
The eigenvalues of the monodromy on the Milnor fiber of a sum of powers
were determined in \cite{Pham}, see also \cite{Brieskorn}. We recall the
calculation needed here.

\begin{lemma} 
Let $T_{\vec{c}}$ denote the monodromy on the reduced middle homology of the Milnor fiber $M_{\vec{c}}$ of $\sum_i x_i^{c_i}=t$, and let $\chi_{\vec{c}}$ be its characteristic polynomial. 
Then
$$
\chi_{\vec{c}}(1)=\pm1.
$$
\end{lemma}

\begin{proof}
For $x^r=t$, the $\mu_r$-action $x\mapsto e^{2\pi i/r} x$ equals the monodromy action, whose eigenvalues on reduced homology are
$$
\exp(2\pi i a) \qquad \text{with} \qquad a=
\textstyle{\frac{1}{r}},\ldots,\textstyle{\frac{r-1}{r}}.
$$
By the Thom--Sebastiani theorem, the $\mu_c$-action on the homology of
$M_{\vec{c}}$ equals the monodromy action, and its eigenvalues are
$$
\exp(2\pi i a) \qquad \text{with} \qquad a=\textstyle{\sum_{i=1}^{n}\frac{b_i}{c_i}} \qquad \text{and} \qquad b_i=1,\ldots,c_i-1.
$$ 
Since the $c_i$ are pairwise coprime, these eigenvalues are pairwise distinct and  none of these is a $c_i$-th root of unity for any $i$.
Therefore $\chi_{\vec{c}}(\lambda)$ divides
$$
\frac{\lambda^c-1}{\lambda-1}\prod_i\frac{\lambda-1}{\lambda^{c_i}-1}.
$$
The value of this expression at $\lambda=1$ is $1$, which proves the claim.
\end{proof}

It is most economical to compactify $Z_{\vec c}^\circ$ in a weighted projective space.
Let
$$
Z_{\vec{c}}\coloneq \left\{\textstyle\sum_{i=1}^n x_i^{c_i}=tw^c\right\} \subset\bP(a_1,\ldots,a_n,1)\times\bA^1_t,
$$
where $x_i$ has weight $a_i$ and $w$ has weight $1$.
The automorphism $\gamma_{\vec c}$ extends to 
$$
\gamma_{\vec{c}}\colon ([x_1:\cdots:x_n:w],t) \longmapsto ([\zeta^{a_1}x_1:\cdots:\zeta^{a_n}x_n:w],t).
$$
In the chart $w\neq 0$, the central fiber has the unique singularity $\sum_i x_i^{c_i}=0$, and the other fibers are smooth. 
Along $w=0$ the fibers are singular, but they admit a $\mu_c$-equivariant simultaneous resolution. 
Denote the resulting family and induced automorphism by
$$ 
\pi_{\widetilde Z_{\vec{c}}}\colon \widetilde Z_{\vec{c}}\longrightarrow \bA^1_t \qquad \text{and}\qquad \gamma_{\vec{c}}\colon \widetilde Z_{\vec{c}} \longrightarrow \widetilde Z_{\vec{c}},
$$
respectively. 
Thus $\pi_{\widetilde Z_{\vec{c}}}$ has a unique critical point, where
$$
\pi_{\widetilde Z_{\vec{c}}}(x_1,\ldots,x_n)=\textstyle\sum_i x_i^{c_i}, \qquad \gamma_{\vec{c}}(x_1,\ldots,x_n) =(\zeta^{a_1}x_1,\ldots,\zeta^{a_n}x_n).
$$
By \cite[Satz~1]{Brieskorn}, the link of such a singularity is a topological sphere for $n\geq 4$,
so the central fiber of $\pi_{\widetilde Z_{\vec{c}}}$ is a PL-manifold.
However, the link is usually not a standard $C^\infty$-sphere; see \cite[Satz~3]{Brieskorn} for the precise conditions.

\begin{ex}
Take $\vec{c}=(2,3,5)$. Then $c=30$, and in the weighted model
$$
Z_{\vec{c}}= \{x^2+y^3+z^5=tw^{30}\} \subset\bP(15,10,6,1)\times\bA^1_t.
$$
The chart $w\neq0$ is the standard smoothing of the $E_8$ surface
singularity $x^2+y^3+z^5=0$. 
The action is
$$
\gamma_{\vec{c}}([x:y:z:w],t) =([\zeta^{15}x:\zeta^{10}y:\zeta^6z:w],t),
$$
where $\zeta$ is a primitive thirtieth root of unity. 

Let $A$ be an elliptic curve and let $\tau\in \operatorname{Aut}(A)$ denote translation by a point of order $30$.
The special fiber of the quotient family \eqref{def:pi-Y} is singular along
$[0:0:0:1]\times A/\langle\tau\rangle$ and is locally the product of
$A/\langle\tau\rangle$ with the $E_8$ surface singularity.
For the $E_8$ singularity, the Milnor lattice is the $E_8$ lattice and
the Milnor monodromy is the corresponding Coxeter transformation
\cite{Pham,Brieskorn}. Since $T-I$ is unimodular, the resulting family
is a $\Z$-homology fiber bundle but it is not smooth.

The special fiber of the resulting family in Theorem \ref{thm:CK-criterion} is a threefold with canonical singularities, locally analytically isomorphic to the product of an $E_8$ surface singularity with an elliptic curve.
\end{ex}

\begin{ex} \label{PL.man.exmp} 
Adding an even number of squares does not change the monodromy action on the reduced homology of the Milnor fiber, cf.~\cite[Lemmas 1, 3, 4]{Brieskorn}.
Thus for any $n\geq 1$, the Milnor fiber $M$ of
$$
Z\coloneq \{x_1x_2+\cdots+ x_{2n-1}x_{2n}+y^2+z^3+t=0\}\subset \mathbb A^{2n+2}\times \mathbb A^1_t
$$
satisfies $H_{2n+1}(M,\Z)=\Z^2$, and the  $\Z/6$-action induced by
$$
\gamma\colon (x_1,\dots ,x_{2n},y,z, t)\longmapsto
(\zeta^3x_1\dots ,\zeta^3x_{2n},\zeta^3y,\zeta^2z,t)
$$
has characteristic polynomial $x^2-x+1$, where $\zeta$ is a primitive sixth root of unity. 

The central fiber $Z_0$ has an isolated terminal singularity at the origin.
Let $B_\epsilon\subset \mathbb A^{2n+2}$ be an analytic ball of radius $\epsilon>0$ and consider the neighborhood $U_\epsilon=Z_0\cap B_\epsilon$.
For sufficiently small $\epsilon$, \cite[Theorem 2.10]{milnor-book} implies that $U_\epsilon$ is homeomorphic to the cone $C(\partial U_{\epsilon})$ over the link $\partial U_\epsilon$ of $0\in Z_0$ and the homeomorphism $U_\epsilon \cong C(\partial U_\epsilon)$ is a diffeomorphism outside the singular point of $Z_0$.
By \cite[Satz 1]{Brieskorn}, $\partial U_\epsilon$ is homeomorphic to a sphere $S^{4n+1}$. 
Hence, $Z_0$ is a PL-manifold.

If $n=1$, then the given differentiable structure on $\partial U_\epsilon$ must be the standard structure on $S^5$ by \cite{Kervaire-Milnor}.
We can thus endow $U_\epsilon$ with the standard differentiable structure of the ball and get a $C^\infty$-manifold structure on $Z_0$ which agrees with the one induced by the algebraic structure of $Z_0$ outside the singularity. 
\end{ex}

\section{Fibrations over Riemann surfaces} \label{sec:smooth-fibration}

In this section we prove the following result of independent interest, which will be used in the proof of Theorem \ref{thm:main.invis.thm}.

\begin{proposition}\label{prop:smooth-fibration-over-Sigma}  
Let $g\colon X\to \Sigma$ be a proper holomorphic map between connected complex manifolds with the following properties:
\begin{enumerate}[(i)]
    \item $\dim_{\C}\Sigma=1$ and $\dim_{\C}X\geq 4$;
    \item $R^i g_\ast\Z$ is a local system for all $i$;
    \item the fibers of $g$ are simply connected.
\end{enumerate} 
Then, for any analytic neighborhood $U\subset \Sigma$ of the critical values 
of $g$, there is a proper $C^\infty$-submersion $\tilde g \colon X\to \Sigma$, which agrees with $g$ outside of $g^{-1}(U)$ and which is homotopic to $g$.
\end{proposition}

Before we prove the proposition, we formulate the following version of Smale's h-cobordism theorem for manifolds with boundary. 

\begin{theorem}[Relative h-cobordism Theorem]  \label{thm:h-cobordism}
Let $I_x,I_y$ denote intervals $[0,1]$ with coordinates $x$ and $y$. 
Let $C$ be a compact simply connected smooth manifold with corners of dimension at least six and with a proper surjective map $g\colon C\to I_x\times I_y$.
Assume that the boundary of $C$ has four faces $M_0$, $M_1$, $V_0$, $V_1$ corresponding to $y=0,$ $y=1$, $x=0$, $x=1$, respectively.
Suppose that $g$ is a submersion over $y=0,$ $x=0$, and $x=1$, with smooth fiber $F$ and that there are compatible bundle trivializations
\begin{align} \label{eq:boundary-trivialization}
F\times I_x \longrightarrow{\sim} M_0
\qquad \text{and}\qquad 
F\times I_y \longrightarrow{\sim} V_j
\quad \text{for}\quad j=0,1.
\end{align}
If the inclusions $M_i\hookrightarrow C$ are homotopy equivalences for $i=0,1$, then the boundary identifications \eqref{eq:boundary-trivialization} extend to a diffeomorphism
\[
\Psi\colon F\times I_x \times I_y\longrightarrow{\sim} C.
\]
\end{theorem}
\begin{proof}
The result follows from the more general version of the s-cobordism theorem for manifolds with boundary, see \cite[Essay~I, \S 1, p.~4]{KirbySiebenmann}.
\end{proof}

\begin{proof}[Proof of Proposition \ref{prop:smooth-fibration-over-Sigma}]
The critical values of $g$ form a closed discrete subset of $\Sigma$.
It thus suffices to treat the case when $g$ has only one critical value and $U\subset \Sigma$ is a small disc containing the critical value.
Let $I_x=[0,1]$ and $I_y=[0,1]$ be intervals with coordinates $x$ and $y$, respectively.
Choose a closed embedding
$$
Q\coloneq I_x\times I_y\hookrightarrow U\subset \Sigma
$$
of a square which contains the critical value of $g$ in its interior. 
In particular, $g$ is a submersion over a neighborhood of the boundary $\partial Q$. 

Put
\begin{align} \label{def:C}
C\coloneq g^{-1}(Q)\qquad \text{and}\qquad M_i\coloneq g^{-1}(I_x\times\{i\})\quad \text{for $i=0,1$},
\end{align}
and
\begin{align} \label{def:V}
V\coloneq g^{-1}\bigl((\{0\}\times I_y)\cup(\{1\}\times I_y)\bigr).
\end{align}
Since $g$ is a submersion over a neighborhood of $\partial Q$, Ehresmann's lemma allows us to fix compatible bundle trivializations 
\begin{align} \label{eq:trivialize-boundary}
M_0\simeq F\times I_x \qquad\text{and}\qquad V\simeq (F\times I_y)\sqcup(F\times I_y),
\end{align}
where $F$ is a fixed smooth fiber of $g$.

Note that $C$ is a manifold with corners.
It is a cobordism between the manifolds with boundary $M_0$ and $M_1$.
The idea of the proof is to apply the h-cobordism theorem (see Theorem \ref{thm:h-cobordism}). 
This trivializes $C$ but the trivialization $\Psi$ might not be compatible with the natural projection to $I_x$ at the boundary component $y=1$.
To fix this, we will apply Cerf's theorem \cite{Cerf}.
Once $\Psi$ is compatible with the projection to $I_x\times I_y$ over (a neighborhood of) the boundary of $I_x\times I_y$, we can define $\widetilde g$ via the diffeomorphism $\Psi$ and the projection $F\times I_x\times I_y\to I_x\times I_y$, which is clearly a submersion. 

\medskip

\textbf{Step 1.} The inclusion $M_i\hookrightarrow C$ is a homotopy equivalence for $i=0,1$.
Moreover, $C$ is simply connected.

\begin{proof}
The argument is similar to the proof of Lemma \ref{cor:whitehead}.
Indeed, since $g$ is a $\Z$-homology fiber bundle, and $I_x$ and $Q$ are contractible, $M_i\hookrightarrow C$ induces an isomorphism on homology, hence is a homotopy equivalence by Whitehead's theorem \cite[Corollary 4.33]{H}.
Here we used that $M_i$ and $C$ are simply connected, because the fibers of $g$ are simply connected, see \cite[Main Theorem]{Smale}.
\end{proof}

\textbf{Step 2.} There is a diffeomorphism
$$
\Psi\colon (F\times I_x)\times I_y\longrightarrow C
$$
which agrees with the chosen trivialization \eqref{eq:trivialize-boundary} on
$$
\bigl((F\times I_x)\times\{0\}\bigr) \ \cup\ \bigl((F\times\{0,1\})\times I_y\bigr).
$$

\begin{proof}
By Step 1, $C$ is a h-cobordism between the manifolds with boundary $M_0$ and $M_1$.
Since $\dim_{\R}C=2\cdot \dim_{\C}X\geq 8$, this h-cobordism is trivial by Smale's h-cobordism theorem with boundary, see Theorem \ref{thm:h-cobordism}.
This concludes Step 2. 
\end{proof}

Since $g$ is a submersion over the top edge $I_x\times \{1\}$ as well, we can fix by Ehresmann's lemma a bundle trivialization 
$$
\theta \colon F\times I_x\longrightarrow M_1 \qquad \text{with}\qquad  g(\theta (z,x))=(x,1),
$$
which agrees with $\Psi$ from Step 2 at $x=0$. 
We then consider the composition
$$
\Phi \coloneq \theta ^{-1}\circ \Psi|_{(F\times I_x)\times\{1\}}\colon F\times I_x\longrightarrow F\times I_x ,
$$
which is the identity at $x=0$.
\medskip

\textbf{Step 3.} There is a smooth path of pseudo-isotopies
$$
\Phi_t\colon F\times I_x \longrightarrow F\times I_x
$$
with $\Phi_0=\Phi$ and $\Phi_1=\id$.

\begin{proof}
Cerf \cite[p.~7, D\'efinition 2]{Cerf} considers the group
$$
P(F)\coloneq \{\varphi \in \operatorname{Diff}(F\times I_x)\mid \varphi|_{F\times \{0\}}=\id\} 
$$
of pseudo-isotopies, equipped with the $C^\infty$-topology. 
Note that $F$ is a simply connected compact manifold without boundary of real dimension at least $6$.
Hence, \cite[p.~8, Théorème 0]{Cerf} implies that $P(F)$ is connected and there is a smooth path $\Phi_t$ with the required properties.
This concludes Step 3.
\end{proof}

\textbf{Step 4.}
There is a diffeomorphism
$$
\Psi'\colon (F\times I_x)\times I_y\longrightarrow C
$$
which agrees with the chosen trivializations~\eqref{eq:trivialize-boundary} on
$$
\bigl((F\times I_x)\times\{0\}\bigr)
\cup
\bigl((F\times\{0\})\times I_y\bigr),
$$
is fiber-preserving on the right vertical face, in the sense that
$$
g\bigl(\Psi'(z,1,y)\bigr)=(1,y)
\qquad
\text{for all }(z,y)\in F\times I_y,
$$
and satisfies
$
\Psi' |_{(F\times I_x)\times\{1\}}=\theta .
$
\begin{proof}
Define
$$
\Xi\colon (F\times I_x)\times I_y
   \longrightarrow (F\times I_x)\times I_y,
\qquad
\bigl((z,x),y\bigr)\longmapsto \bigl(\Phi_{1-y}^{-1}(z,x),y\bigr),
$$
and use it to modify $\Psi$ from Step 2 as follows:
$$
\Psi'=\Psi\circ\Xi.
$$
Since $\Phi_1=\id$, the map $\Psi'$ agrees with $\Psi$ on the
bottom face $I_x\times \{0\}$. 
Since every $\Phi_t$ fixes $F\times\{0\}$ pointwise, it also agrees with $\Psi$ on the left vertical face. 
Moreover, every $\Phi_t$ preserves $F\times\{1\}$, and hence $\Psi'$ is
fiber-preserving on the right vertical face $\{1\}\times I_y$. 
Finally, using $\Phi_0=\Phi$, we obtain
$$
 \Psi' |_{(F\times I_x)\times\{1\}}
 =
 \Psi |_{(F\times I_x)\times\{1\}}\circ\Phi^{-1}
 =\theta .
$$
This proves Step 4.
\end{proof}
 
By Step 4, we have
\[
g\bigl(\Psi'(z,x,y)\bigr)=(x,y)
\]
for all $(x,y)\in \partial Q$ and $z\in F$.
After modifying $\Psi'$ by a diffeomorphism supported in a collar of
the boundary and equal to the identity on the boundary, we may assume that the same holds in some neighborhood of $\partial Q$.
We may then define
$$
\tilde g_Q\colon C\longrightarrow Q \qquad \text{by}\qquad  \tilde g_Q\bigl(\Psi'(z,x,y)\bigr)=(x,y).
$$
This is a smooth submersion and agrees with $g$ on a neighborhood of
$\partial C$. 

Since $Q=[0,1]^2$ is convex, 
$$
H\colon C\times [0,1]\longrightarrow Q,\qquad (c,t)\mapsto H_t(c)\coloneq \bigl((1-t)\cdot g(c)
+t\cdot \tilde g_Q(c)\bigr)
$$
defines a homotopy from $g|_C$ to $\tilde g_Q$. 
On the neighborhood of $\partial C$ where $g=\tilde g_Q$, this homotopy is constant. 
Thus $g|_C$ and $\tilde g_Q$ are homotopic relative to that neighborhood.
The map $g$, restricted to a small neighborhood of $X\setminus C$, thus glues with $\tilde g_Q$ to produce a $C^\infty$-submersion $\tilde g\colon X\to \Sigma$ which is homotopic to $g$ and agrees with $g$ in a neighborhood of $X\setminus g^{-1}(Q)\supset X\setminus g^{-1}(U)$.
Since $g^{-1}(Q)$ and $Q$ are compact, $\tilde g$ is proper.
This concludes the proof.
\end{proof}

\begin{question}\label{question:higher-dimensional-targets}
To what extent does Proposition~\ref{prop:smooth-fibration-over-Sigma} extend to higher-dimensional targets?
\end{question}

\section{Proof of Theorem \ref{thm:main.invis.thm}} \label{sec:main.invis.thm}

\subsection{Submersions and homology fiber bundles under blow-ups}

In this section, $f\colon M\to N$ denotes a proper submersion between complex manifolds and
$$
\pi\colon M'\coloneq \Bl_Z M\longrightarrow M
$$ 
denotes the blow-up of $M$ along a closed connected complex submanifold $Z\subset M$  of codimension $c\geq 2$.
We denote the exceptional divisor of $\pi$ by
$$
E\coloneq \bP(N_{Z/M})\subset M' \qquad \text{with projection}\qquad p\colon E\to Z.
$$
The natural maps induced by $f$ are denoted by
$$
f' \colon M'\longrightarrow N\qquad \text{and} \qquad f_Z\colon Z\longrightarrow N .
$$ 
With this notation, we have the following simple but useful result.

\begin{proposition} \label{prop:blow-up}
In the above notation, the following holds:
\begin{enumerate}[(i)]
    \item $f'$ is a submersion if and only if $f_Z$ is a submersion. \label{item:prop:blow-up:1}
    \item $f'$ is a $\Z$-homology fiber bundle if and only if $f_Z$ is a $\Z$-homology fiber bundle.\label{item:prop:blow-up:2}
\end{enumerate}
\end{proposition}
\begin{proof}
Let us first prove item \ref{item:prop:blow-up:1}.
 If $f_Z$ is a submersion, then $f_Z$ is flat and the fibers of $f_Z$ are smooth.
It follows that the fibers of $f'$ are smooth, hence $f'$ is a submersion.
Conversely, assume that $f'$ is a submersion and let $\widetilde z=(z,[\ell])\in E=\mathbb P(N_{Z/M})$.
Then
$$
\operatorname{im}(d\pi_{\widetilde z}) = T_zZ+\widetilde\ell,
$$
where $\widetilde\ell\subset T_zM$ is a lift of $\ell$.
Suppose that $f_Z$ is not a submersion at $z$. 
Choose a hyperplane $H\subset T_{f(z)}N$ containing $df_z(T_zZ)$. Since $df_z$ is surjective and $\operatorname{codim}_M Z\geq 2$, there is a line $\widetilde\ell\subset df_z^{-1}(H)$ which is not contained in $T_zZ$. 
At the corresponding point $\widetilde z\in E$, the image of  $df'_{\widetilde z}$ is contained in $H$.
Hence $f'$ is not a submersion, proving \ref{item:prop:blow-up:1}.

For an open subset $U\subset N$, set
$M_U=f^{-1}(U)$, $Z_U=Z\cap M_U$, $M'_U=(f')^{-1}(U)=\operatorname{Bl}_{Z_U}M_U$.
The integral blow-up formula gives natural isomorphisms
\[
H^k(M'_U,\mathbb Z)
\cong
H^k(M_U,\mathbb Z)
\oplus
\bigoplus_{i=1}^{c-1}H^{k-2i}(Z_U,\mathbb Z),
\]
cf.~\cite[Thm.~7.31]{voisin}.
Here, the isomorphism is induced by $\pi_U^*$ together with the
Gysin maps from the exceptional divisor. 
These isomorphisms are compatible with restriction to open subsets of $N$. 
Hence, after sheafification, they give
\begin{equation}\label{eq:blow-up-direct-images}
R^k f'_{*}\mathbb Z
\simeq
R^k f_{*}\mathbb Z
\oplus
\bigoplus_{i=1}^{c-1}R^{k-2i}f_{Z*}\mathbb Z. 
\end{equation}
Since $R^k f_*\mathbb Z$ is a local system by Ehresmann's lemma and $c\geq 2$, formula~\eqref{eq:blow-up-direct-images} immediately implies that $f'$ is a $\mathbb Z$-homology fiber bundle if and only if $f_Z$ is one. 
\end{proof}

\subsection{Promoting homology fiber bundles to homotopy fiber bundles}\label{subsec:promoting}
Let $h\colon V\to \bP^1$ be a morphism of smooth projective varieties which is a $\Z$-homology fiber bundle but not a submersion.
Let $P\to \bP^1$ be a smooth projective morphism with simply connected fibers and with an embedding $V\hookrightarrow P$ which lifts $h$.
Assume $\codim_P V\geq 2$, and put 
\begin{align} \label{def:Y-general}
Y\coloneq \Bl_VP\qquad \text{with natural map}\qquad g\colon Y\longrightarrow \bP^1 .
\end{align}

\begin{proposition} \label{prop:X-basic-1}
In the above notation, the following holds.
\begin{enumerate}[(i)]
\item All fibers of $g\colon Y\to \bP^1$ are simply connected.\label{item:X-basic-1:0}
\item For all $i$, $R^ig_\ast \Z$ is a trivial local system on $\bP^1$.
\label{item:X-basic-1:2.5}
\item $g$ is a homotopy fiber bundle.
\label{item:X-basic-1:3}
\item $g$ is not a submersion.\label{item:X-basic-1:4}
\end{enumerate}
\end{proposition}
\begin{proof}
Fix $y\in \bP^1$ and choose a sufficiently small neighborhood $U\subset \bP^1$ of $y$ such that $g^{-1}(y)\hookrightarrow g^{-1}(U)$ and $P_y\hookrightarrow P_U=P\times_{\bP^1}U$ are deformation retracts.
Since $P_y$ is simply connected, the same holds for $P_U$ and hence for the blow-up $g^{-1}(U)$. 
Thus, $g^{-1}(y)$ is also simply connected. 
This proves item \ref{item:X-basic-1:0}.

By Proposition \ref{prop:blow-up}, $g$ is a $\Z$-homology fiber bundle but not a submersion.
Since $\bP^1$ is simply connected, each $R^ig_\ast \Z$ is trivial.
This proves \ref{item:X-basic-1:2.5} and \ref{item:X-basic-1:4}.
Item \ref{item:X-basic-1:3} follows from Lemma~\ref{cor:whitehead}.
\end{proof}

\subsection{Explicit examples and proof of Theorem \ref{thm:main.invis.thm}}

Let $L=\cO_{\bP^1}(1)$, and let $s$ be a general section of $L^{\otimes 6}$. 
Define the surface 
\begin{align*} %\label{def:S}
S\coloneqq \{zy^2+x^3+sz^3=0\}\subset \bP(\cO_{\bP^1}\oplus L^{\otimes 2}\oplus L^{\otimes 3}),
\end{align*}
where the coordinates $x,y,z$ correspond to the summands $L^{\otimes 2}$, $L^{\otimes 3}$, and $\mathcal O_{\bP^1}$, respectively.
This is a rational elliptic surface with six singular fibers of Kodaira type~$\mathrm{II}$.

Around each zero of $s$, the projection $\pi\colon S\to \bP^1$ is analytically locally isomorphic to the local model from \eqref{def:mathcal-X-CK}. 
We have an order-six fiberwise automorphism on $\bP(\cO_{\bP^1}\oplus L^{\otimes 2}\oplus L^{\otimes 3})$: 
$$
\gamma \colon  [x: y: z]\mapsto [\zeta^2 x: \zeta^3 y: z],
$$
where $\zeta$ denotes a primitive sixth root of unity.
Let $A$ be an abelian variety and  $\tau\in \operatorname{Aut}(A)$ be given by translation by a point of order six.
We define
$$
V\coloneq (S\times A)/ \langle (\gamma,\tau) \rangle ,
$$
which admits a finite \'etale cover $S\times A\to V$ of degree $6$.
Consider the natural morphism
$$
h\colon V\longrightarrow \bP^1
$$
that is induced by $S\to \bP^1$.
By Theorem \ref{thm:CK-criterion} and Example \ref{ex:CK-cubic}, this map is a $\bZ$-homology fiber bundle, but not a submersion.

From now on, we specialize to the case where $A$ is an elliptic curve.
Our aim is to lift $h\colon V\to \bP^1$ to an embedding $V\hookrightarrow P$, for some smooth projective morphism $P\to \bP^1$ with simply connected fibers.
By generic projection, we can  choose $P\coloneq \bP^N\times \bP^1$ for any $N\geq 6$, but with more work we can get such a $P\to \bP^1$ of relative dimension four, as we show next. 

\begin{ex} \label{ex:RC-4-fold-bundle}
Let $T$ be the degree 6 del~Pezzo surface obtained by blowing up the three coordinate vertices of $\bP^2$.
The standard Cremona transformation $[x_0:x_1:x_2]\mapsto [x_0^{-1}:x_1^{-1}:x_2^{-1}]$ lifts to an order 2 automorphism  $\tau_2$ of $T$, and  $[x_0:x_1:x_2]\mapsto [x_1:x_2:x_0]$ lifts to an order 3 automorphism $\tau_3$ of $T$. 
Since $\tau_2$ and $\tau_3$ commute, $\tau\coloneq \tau_2\tau_3^{-1}$ is an order 6 automorphism of $T$.

The action is free away from the fixed points of $\tau_2$ and of $\tau_3$. 
These are $[\pm 1: \pm 1: \pm1 ]$ and $[1:\epsilon:\epsilon^2]$ where $\epsilon$ is a third root of unity.

Let $A\subset T$ be the  curve  $\{\sum_{i\neq j}x_ix_j^2=0\}$.
Then $A$ is a $\tau$-invariant, smooth, elliptic curve which does not contain any of the $\tau_2$ or $\tau_3$ fixed points. 
Thus $\tau$ restricts to a translation  of order six on $A$.

Consider
$$
V\coloneq (S\times A)/ \langle (\gamma,\tau) \rangle
\subset
\bigl(\bP(\cO_{\bP^1}\oplus L^{\otimes 2}\oplus L^{\otimes 3})\times T\bigr)/ \langle (\gamma,\tau) \rangle \eqqcolon P'.
$$
Here $P'\to\bP^1$ is a fiber bundle with fiber
$$
R'\coloneq  \bigl(\bP^2\times T\bigr)/ \langle (\gamma,\tau) \rangle.
$$
Let $R\to R'$ be a functorial resolution of singularities.
It extends to a fiber bundle $P\to\bP^1$ with fiber
$R$, such that the projection $P\to P'$ is an isomorphism over the smooth locus of $P'$.
Since $R$ is unirational, $P\to \bP^1$ has simply connected fibers.

In fact, $R$ is rational because $R\to \bP^2/\langle \gamma\rangle$ admits a rational section, given by $[1:1:1]$ and so its generic fiber is a del Pezzo surface of degree 6 with a rational point, hence is rational, see \cite[Theorem 9.4.8(iii)]{Poonen}.
Moreover, $\bP^2/\langle \gamma\rangle$ is rational and so is $R$.
The same argument shows that the generic fiber of $P\to \bP^1$ is rational.
In particular, $P\to \bP^1$ is birational to the projection $\bP^4\times \bP^1\to \bP^1$.

Note that $P'$ is singular exactly at the images of the fixed points of the square or cube of  $(\gamma,\tau)$, hence $P'$ is smooth along the image of $V$. 
Thus the embedding  $V\hookrightarrow P'$ lifts to an embedding $V\hookrightarrow  P$.  
The resulting $g\colon Y\to \bP^1$ from \eqref{def:Y-general} has then $\dim Y=5$.
\end{ex}

With the above example at hand, we are finally able to prove Theorem \ref{thm:main.invis.thm} stated in the introduction.

\begin{proof}[Proof of Theorem \ref{thm:main.invis.thm}]
We define
\begin{align} \label{def:Y}
Y\coloneq \Bl_VP\qquad \text{with natural map}\qquad g\colon Y\longrightarrow \bP^1 ,
\end{align}
where $V$ and $P$ are as in Example \ref{ex:RC-4-fold-bundle}; in particular, $\dim Y=5$.

Item \ref{item:main.invis.thm:2}, the homotopy-equivalence and
simple-connectivity assertions in item \ref{item:main.invis.thm:3},
and the non-smoothness assertion in item \ref{item:main.invis.thm:5}
follow from Proposition~\ref{prop:X-basic-1};
item \ref{item:main.invis.thm:4} follows from
Proposition~\ref{prop:smooth-fibration-over-Sigma}.

It remains to show that the fibers of $g$ are PL-manifolds and that
$g$ is not a $C^0$-fiber bundle. 
In local analytic coordinates, the singular fibers of $g$ are obtained by blowing up
$$
\{x^2+y^3=u=0\}\subset \bA^4_{xyuv}.
$$ 
After blow-up, the chart  
\begin{align} \label{eq:blow-up-sing}
\{x^2+y^3-u'w=0\}\subset \bA^5_{xyu'vw}
\end{align}
given by $u=u'w$ contains the singularities.

Let $Y_0$ be a singular fiber of $g$.
Since $Y_0$ is algebraic, it admits a subanalytic triangulation, unique up to PL homeomorphism by \cite[Corollary~4.3]{Shiota-Yokoi}.
As we noted in Example~\ref{PL.man.exmp}, the link of every point in \eqref{eq:blow-up-sing} is a standard PL sphere by \cite[Satz 1]{Brieskorn}.
This shows that $Y_0$ is locally around each of its singular points a PL manifold; uniqueness up to PL homeomorphisms then shows that these local structures glue to a global PL-manifold structure on $Y_0$.  

Note finally that $g$ has six singular fibers, so is not a submersion.
To see that it is not a $C^0$-fiber bundle, note that the Milnor fiber $M$ of \eqref{eq:blow-up-sing} satisfies $H_3(M,\Z)\cong \mathbb Z^2$ (see Example~\ref{PL.man.exmp}) and so there are local vanishing cycles. 
This implies that $g$ is not a $C^0$-fiber bundle, thus completing the proof of the theorem.
\end{proof}

\begin{remark} \label{rem:terminal}
The  computation leading to \eqref{eq:blow-up-sing} also shows that, locally at each critical point, $g\colon Y\to \bP^1$ can be written as $ (x_1,\dots ,x_5)\mapsto x_1^2+x_2^2+x_3^2+x_4^3$ in suitable analytic local charts on $Y$.
The singularities of the fibers are locally given by $\{x_1^2+x_2^2+x_3^2+x_4^3=0 \}\subset \mathbb A^5$, which are terminal.
\end{remark}

\begin{remark}
One can show that the singular fibers of $g\colon Y\to \bP^1$ produced in the proof of Theorem \ref{thm:main.invis.thm} admit  a $C^\infty$-manifold structure that agrees with the one induced by the algebraic structure outside small neighborhoods of the singular locus.

The surface $S$ used in the proof of Theorem \ref{thm:main.invis.thm} (cf.~Example \ref{ex:RC-4-fold-bundle}) can be replaced by any of the other examples from Section \ref{subsec:further-examples}.
In those examples, the singularities of $Z_{\vec{c}}^\circ$ are completely determined by \cite{Brieskorn}, but even if the singularities admit $C^\infty$-structures, it is not clear that there is a natural one. 
\end{remark} 

\begin{question}\label{question:homeomorphic-fibers}
  Are there examples  as in Theorem \ref{thm:main.invis.thm} where
  \begin{enumerate}[(i)]
  \item all fibers are PL-homeomorphic to each other, or
  \item the dimension of $Y$ is 4, or
  \item all fibers have nef canonical class?
  \end{enumerate}
\end{question}

\subsection{Three-dimensional smoothness conjecture for simply connected fibers} \label{subsec:thm:3-folds}
Here we show that counterexamples to the smoothness conjecture in \cite{FdBK}
do not exist in dimension three.

\begin{theorem} %[\cite{HS21}]
\label{thm:3-folds}
Let $Y$ be a smooth, complex, projective threefold and  $g\colon Y\to \bP^1$  a morphism with simply connected fibers.
Assume that $g$ is a $\Q$-homology fiber bundle. 
Then $g$ is smooth.
\end{theorem}
\begin{proof}
Let $E$ be an elliptic curve and  $E\to \bP^1$  a double cover whose critical values are disjoint from the critical values of $g$.
Let $X\coloneq Y\times_{\bP^1}E$ with natural map $f\colon X\to E$.
Since $g$ and hence $f$ are $\Q$-homology fiber bundles, the decomposition theorem (see \cite{BBD}) implies $Rf_\ast \Q_X=\oplus R^if_\ast \Q_X[-i]$.
Since $f$ is a base change of $g$ and $\bP^1$ is simply connected, $R^if_\ast \Q_X\cong \Q_E^{\oplus r_i}$ is a trivial local system for all $i$.
It follows that $H^\ast(X,\Q)$ is a free $H^\ast(E,\Q)$-module.
Hence the Aomoto complex $(H^\ast(X,\C),\wedge f^\ast \omega)$ is exact for every nonzero $\omega\in H^0(E,\Omega^1_E)$.
Since $g$ (and hence $f$) has simply connected fibers, $f$ induces an isomorphism $\pi_1(X)\cong \pi_1(E)$.
In particular, any finite \'etale cover $\pi\colon X'\to X$ is induced by a finite \'etale cover $E'\to E$.
Repeating the above argument for $E'$ in place of $E$ then shows that $(H^\ast(X',\C),\wedge \pi^\ast f^\ast \omega)$ is exact.
At this point \cite[Theorem 1.4]{HS} (see also \cite{pietig}) implies that $f^\ast \omega$ has no zeros, and so $f$ is smooth.
This forces $g$ to be smooth as well.
\end{proof}

\begin{remark}
Let $X$ be a compact K\"ahler manifold and let $\omega\in H^0(X,\Omega^1_X)$.
It was asked in \cite[Remark 1.7]{S} whether $\omega$ has no zeros if for each finite \'etale cover $\pi\colon X'\to X$, the Aomoto complex $(H^\ast(X',\C),\wedge \pi^\ast  \omega)$ is exact.
A positive answer to this question is given in dimensions 2 and 3 in \cite{S,HS,pietig}.
In the companion paper \cite{Aomoto}, we will use similar constructions as used in this paper to give a negative answer to this question in each dimension at least 5.
\end{remark}

\section{Kotschick's conjecture and a question of Koll\'ar--Pardon} \label{sec:applications}

Let $X$ be a compact Kähler manifold. 
In 2013, Kotschick asked whether the following two conditions are equivalent, cf.~\cite[Question 15]{K}:
\begin{kotschickconditions}[series=kotschick]
\item $X$ admits a holomorphic one-form without zeros;\label{cond:A}
\item $X$ admits a real closed one-form without zeros.\label{cond:B}
\end{kotschickconditions}
Kotschick observed that \condref{A} implies \condref{B} and asked if the converse holds.
By Tischler's theorem \cite{T}, Property \condref{B} is equivalent to  
\begin{enumerate}
\item[(B')] $X$ is a smooth fiber bundle over $S^1$. 
\end{enumerate}

In the literature, the equivalence of \condref{A} and \condref{B} became known as Kotschick's conjecture.
It has been verified in dimensions $2$ and $3$ in \cite{K,S} and \cite{HS,pietig}, respectively.
For closely related work on holomorphic versions of Tischler's theorem, sparked by the above conjecture, see e.g.~\cite{DHL,SY,hao,C1,C2,HWZ,C3,pietig-kodaira,DHLW}. 

As a consequence of Theorem \ref{thm:main.invis.thm}, we prove the following.

\begin{theorem} \label{thm:main2} 
There is a smooth complex projective variety $X$ of dimension $5$ and with $b_1(X)=2$, such that the Albanese morphism $f\colon X\to {\rm Alb}(X)$ has the following properties:
\begin{enumerate}[(i)]
\item $f$ is a homotopy fiber bundle;
\item there is a $C^\infty$-fiber bundle structure $\tilde f\colon X\to {\rm Alb}(X)$ which is homotopic to $f$;
\item $f$ is not a submersion and any holomorphic one-form on $X$ has a zero. 
\end{enumerate}
\end{theorem}
\begin{proof}
Let $E$ be an elliptic curve and $E\to \bP^1$  a surjective morphism whose critical values are disjoint from the critical values of $g\colon Y\to \bP^1$ from \eqref{def:Y}.
We then define  
\begin{align} \label{def:X-2}
X\coloneq Y\times_{\bP^1}E\qquad \text{with natural map}\qquad  f\colon X\longrightarrow E.
\end{align} 
Note that $X$ is smooth because $g$ is smooth in a neighborhood of the critical values of $E\to \bP^1$.
Moreover, $f$ is a homotopy fiber bundle with simply connected fibers since $g$ is.  
Similarly, $f$ is not a submersion because $g$ is not. 

Since $f$ has simply connected fibers, 
$f$ induces an isomorphism on fundamental groups.
Hence $\Alb(X)\cong E$ and $f$ is the Albanese morphism.
Since $f$ is not a submersion and $\Alb(X)$ is an elliptic curve, each holomorphic one-form on $X$ has a zero. 

By Proposition~\ref{prop:smooth-fibration-over-Sigma}, there is a $C^\infty$-fiber bundle structure $\tilde f\colon X\to E$ which is homotopic to $f$.
This concludes the proof of the theorem. 
\end{proof}

Theorem \ref{thm:main2} yields the following strong counterexample to Kotschick's conjecture.

\begin{corollary}\label{cor:Kotschick}
Condition \condref{B} does not imply condition \condref{A}.
In fact, there is a projective manifold $X$ of dimension $5$ with $b_1(X)=2$ such that each nonzero cohomology class $u\in H^1(X,\R)$ can be represented by a closed real one-form without zeros, while the harmonic representative of $u$ always has zeros.
\end{corollary}

\begin{proof}
Let $X$ be as in \eqref{def:X-2}.
Then $b_1(X)=2$ and $f\colon X\to E$ is the Albanese morphism of $X$. 
By Theorem \ref{thm:main2}, there is a $C^\infty$-fibration ${\tilde f}\colon X\to \Alb(X)$ which is homotopic to the Albanese morphism $f$.
Hence, ${\tilde f}^\ast\colon H^1(\Alb(X),\R)\to H^1(X,\R)$ is an isomorphism.
Since ${\tilde f}$ is a submersion, it follows that every nonzero class in $H^1(X,\R)$ can be represented by a closed real one-form without zeros. 

Now let $\alpha_u\in A^1(X,\R)$ be the harmonic representative of a nonzero class $u\in H^1(X,\R)$ and set $\omega_u=(\alpha_u)^{1,0}$. 
Since $X$ is K\"ahler and $u\neq 0$, $\omega_u$ is a nontrivial holomorphic one-form. 
By Theorem \ref{thm:main2}, $\omega_u$ vanishes at some point, where $\alpha_u=\omega_u+\overline{\omega_u}$ also vanishes. 
This concludes the proof of Corollary \ref{cor:Kotschick}.
\end{proof}

\begin{remark}
 The variety in Theorem \ref{thm:main2} has Kodaira dimension $-\infty$, since the general fiber of its Albanese map is rational. 
We can also construct examples of dimension $n\geq 7$ and any intermediate Kodaira dimension $\kappa\in \{ 0,1,\dots ,n-1\}$, by replacing the  variety $P$ in \eqref{def:Y} by a product $Bl_p H\times \bP^1$.
Here $H$ is a any smooth variety of dimension $n-1$, Kodaira dimension $\kappa$, and $p\in H$ is a point.
We may then factor $V\to \bP^1$ from Example \ref{ex:RC-4-fold-bundle} through $D\times \bP^1\to \bP^1$ and hence through $Bl_pH\times \bP^1\to \bP^1$, where $D\cong \bP^{n-2}$ denotes the exceptional divisor of $Bl_pH$. 
The blow-up $Y=Bl_V( Bl_p H\times \bP^1)$ with its natural map to $\bP^1$  yields examples as in Theorem \ref{thm:main.invis.thm} whose fibers have  Kodaira dimension $\kappa$.
\end{remark}
 
By a theorem of Popa and Schnell \cite{PS}, every holomorphic one-form on a smooth complex projective variety of general type has a zero.

\begin{question}
Can a smooth complex projective variety of general type admit a real closed one-form without zeros? Equivalently, can its underlying differentiable manifold fiber smoothly over $S^1$?
\end{question}

\begin{question}\label{question:Kotschick-nef}
Are conditions \condref{A} and \condref{B} from Kotschick's conjecture equivalent for smooth  complex projective varieties with nef canonical class?
\end{question} 

As another consequence of Theorem~\ref{thm:main.invis.thm}, we obtain a negative answer to a question of Koll\'ar and Pardon
\cite[Question 27]{KP} (see also \cite[Question 1.3]{LMW}).

\begin{corollary}\label{cor:KP}  
Let $f\colon X\to E$  be as in \eqref{def:X-2}.
Then $f_*\colon \pi_1(X)\to \pi_1(E)$ is an isomorphism, 
 the universal cover $\widetilde X$ of $X$ has the homotopy type of a finite CW complex, but $f$ is not a
$C^\infty$-fiber bundle.
\end{corollary}

\begin{proof} 
The universal cover of $E$ is  $\widetilde E=\C$. 
For a regular value $y\in\widetilde E$, the inclusion $\widetilde X_y\hookrightarrow\widetilde X$ is an integral homology equivalence by the Leray spectral sequence, since $\widetilde X\to\widetilde E$ is a $\Z$-homology fiber bundle and $\widetilde E$ is contractible. Both spaces are simply connected:
this follows for $\widetilde X$ from the isomorphism $\pi_1(X)\simeq\pi_1(E)$ and for $\widetilde X_y$ from Proposition~\ref{prop:X-basic-1}. 
Lemma~\ref{cor:whitehead}
therefore shows that $\widetilde X_y\hookrightarrow\widetilde X$ is a homotopy equivalence.
As $\widetilde X_y$ is compact, $\widetilde X$ has the homotopy type of a finite CW complex.
Finally, $f$ is not a $C^\infty$-fiber bundle, since it is not
a submersion by Proposition \ref{prop:X-basic-1}. 
\end{proof}

\section*{AI disclosure}
The mathematical content of the paper is due to the authors. 
 ChatGPT was used by one of us to assist with parts of the drafting process. 

\section*{Acknowledgements} 
We are grateful to Javier Fern\'andez de Bobadilla, Feng Hao, Claudio Llosa--Isenrich, Simone Noja, Simon Pietig,  Pierre Py,  and Yang Su for helpful conversations and references.

MC is partially supported by the Universit\`a degli Studi di Bari Aldo Moro and is a member of INdAM-GNSAGA.
Partial  financial support  to JK   was provided  by the Simons Foundation   (grant number SFI-MPS-MOV-00006719-02).
SS received funding from the European Research Council (ERC) under the European Union’s Horizon 2020 research and innovation programme under grant agreement No.~948066 (ERC-StG RationAlgic).
The research was partly conducted in the framework of the DFG-funded research training group RTG 2965: From Geometry to Numbers, Project number 512730679.

%------------------------------------------------------------------
\end{document}